\documentclass{amsart}
\usepackage{amssymb, secteqn}
\usepackage{graphicx}
\usepackage{color}
\definecolor{darkgreen}{cmyk}{1,0,1,.2}
\definecolor{m}{rgb}{1,0.1,1}
\definecolor{green}{cmyk}{1,0,1,0}

\definecolor{test}{rgb}{1,0,0}
\definecolor{cmyk}{cmyk}{0,1,1,0}

\long\def\green#1{\textcolor {green}{#1}}
\long\def\m#1{\textcolor {m}{#1}}

\newtheorem{Equation}{}[section]

\newtheorem{conjecture}[Equation]{Conjecture}
\newtheorem{corollary}[Equation]{Corollary}
\newtheorem{definition}[Equation]{Definition}

\newtheorem{proposition}[Equation]{Proposition}

\newtheorem{theorem}[Equation]{Theorem}

\def\ch{\operatorname{ch}}

\def\codim{\operatorname{codim}}

\def\I{\operatorname{I}}
\def\ind{\operatorname{ind}}

\def\GL{\operatorname{GL}}

\def\oH{\operatorname{H}}

\def\Ind{\operatorname{Ind}}

\def\mod{\operatorname{mod}}

\def\tr{\operatorname{tr}}

\def\span{\operatorname{span}}

\def\C{\mathbb C}
\def\D{\mathbb D}
\def\H{\mathbb H}
\def\N{\mathbb N}
\def\R{\mathbb R}
\def\S{\mathbb S}
\def\T{\mathbb T}
\def\Z{\mathbb Z}

\def\Q{\mathbb Q}

\def\mfa{{\mathfrak a}}

\def\cA{{\mathcal A}}
\def\maA{{\mathcal A}}
\def\maB{{\mathcal B}}

\def\maD{{\mathcal D}}

\def\maE{{\mathcal E}}

\def\cG{{\mathcal G}}

\def\cG{{\mathcal G}}

\def\cH{{\mathcal H}}

\def\maR{{\mathcal R}}

\def\cS{{\mathcal S}}
\def\maS{{\mathcal S}}
\def\maT{{\mathcal T}}

\def\maU{{\mathcal U}}
\def\cU{{\mathcal U}}

\def\what{\widehat}
\def\wtit{\widetilde}

\def\wD{\widetilde{D}} 
\def\wE{\widetilde{E}}
\def\wF{\widetilde{F}}

\def\wL{\widetilde{L}} 
\def\wM{\widetilde{M}}

\def\wB{\widetilde{B}}

\def\dd{\displaystyle}

\def\ep{\epsilon}

\begin{document}



\title[Geometric Non-Commutative Geometry \today] 
{Geometric Non-Commutative Geometry  \\ \today}


\author[M-T. Benameur]{Moulay-Tahar Benameur}
\address{Institut Montpellierain Alexander Grothendieck, UMR 5149 du CNRS, Universit\'e de Montpellier}
\email{moulay.benameur@umontpellier.fr}

\author[J.  L.  Heitsch \today]{James L.  Heitsch}
\address{Mathematics, Statistics, and Computer Science, University of Illinois at Chicago} 
\email{heitsch@uic.edu}

\thanks{MSC (2010) 53C12, 57R30, 53C27, 32Q10. }

\begin{abstract}
In a recent paper, the authors proved that no spin foliation on a compact enlargeable  manifold with  Hausdorff homotopy graph admits a metric of positive scalar curvature on its leaves.  This result extends groundbreaking results of Lichnerowicz, Gromov and Lawson, and Connes on the non-existence of metrics of positive scalar curvature.
In this paper we review in more detail the material needed for the proof of our theorem and we extend our non-existence results to non-compact manifolds of bounded geometry.  
{We also give a first obstruction result for the existence of metric with (not necessarily uniform) leafwise PSC in terms of the A-hat class in Haefliger cohomology.} 
\end{abstract}
\maketitle


\section{Introduction}

In \cite{BH19}, we give a proof of Theorem \ref{BHspin} below, which extends to foliations seminal results on the non-existence of metrics of positive scalar curvature (PSC).   The first result of this type is due to  Lichnerowicz, \cite{Li63}, where he showed that the A-hat genus of a compact spin manifold is an obstruction to such metrics.  This result was greatly expanded in a series of papers by Gromov and Lawson, \cite{GL1,GL2,GL3}, where they introduced the concept of enlargeability, and proved that a compact enlargeable manifold does not admit metrics of PSC.   In \cite{C86}, Connes extended  the Lichnerowicz result to foliations, when he proved  that  the A-hat genus of a compact manifold is an obstruction to  it having a  spin foliation such that the tangent bundle of the foliation admits a metric with PSC.
His proof uses the Connes-Skandalis index theorem for foliations, \cite{CS84},  and Connes' deep result {of topological invariance of the transverse fundamental class.} Our theorem is
\begin{theorem}\label{BHspin}
Suppose that $F$  is a spin foliation with Hausdorff homotopy graph of a compact enlargeable  manifold $M$.  Then $M$  does not admit a metric which induces positive scalar curvature on the leaves of $F$.
\end{theorem}
In  \cite{Z17},  Zhang  showed how to use his results in  \cite{Z16} to remove the Hausdorff assumption.

In {\cite{GL3}, Theorem 7.3, Gromov and Lawson proved  the following.

\begin{theorem}\label{GLNPB}  A spin
enlargeable complete non-compact Riemannian manifold $M$ has no uniform positive lower bound on its scalar curvature. 
\end{theorem}

The existence of such a bound  for PSC metrics (automatic in the compact case) was used in the proof of  Theorem \ref{BHspin}.  We {explain in the present paper  that our proof in \cite{BH19} extends to show} that the scalar curvature on a spin foliation with Hausdorff homotopy graph of a non-compact Riemannian enlargeable  manifold of bounded geometry has no uniform positive lower bound.

  In \cite{SZ18}, Su and Zhang showed how to extend Theorem \ref{GLNPB} to foliations of non-compact (not necessarily complete!) manifolds.  In particular they proved the following.

\begin{theorem} Suppose $F$ is a foliation of an enlargeable Riemannian manifold $M$.   If either $M$ or $F$ is spin, then  the leafwise scalar curvature of the induced metric on $F$ has no uniform positive lower bound.
\end{theorem}

Note that neither theorem precludes the existence of metrics of PSC on $M$ or $F$.
On the other hand, we show how our previous results prove the following new theorem.  See Section \ref{NONPSC}.

\begin{theorem}\label{BHspin2} 
Suppose that $F$  is a spin foliation with Hausdorff homotopy graph of a non-compact Riemannian manifold $M$.  Suppose that $(M,F)$ is of bounded geometry, and that $\what{A}(F) \neq 0$ in $H^*_c(M/F)$, the Haefliger cohomology of $F$.  Suppose further that, for the Atiyah-Singer operator  $\maD$ on $F$ {is transversely regular near $0$ with Novikov-Shubin invariants $NS(\maD) > 3\codim (F)$}.
Then  the metric on the leaves of $F$ does not have positive scalar curvature.

If  $F$ is Riemannian, the condition on the Novikov-Shubin invariant becomes $NS(\maD) > \codim(F)/2$.
\end{theorem}

{Recall that $\maD$ is transversely regular if the spectral projections $P_0$ and $P_\ep=P_{(0, \ep)}$ are transversely smooth, for sufficiently small $\ep$, see \cite{BH08}.}   

The recent results of Zhang  and Su and Zhang, \cite{Z17, SZ18}, using sub Dirac operators, combined with our results, lead us to conjecture the following.

\begin{conjecture}
Let $F$ be a spin foliation with normal bundle $\nu$ of a non-compact manifold $M$.  Suppose that, for the Atiyah-Singer operator on $F$, the associated sub-Dirac operator $\maD$, \cite{Z16}, has Novikov-Shubin invariant $NS(\maD) > 3 \codim(F)$.  Suppose
further that $\dd \int_F  \what{A}(F) \what{L}(\nu) \neq 0$ in
$H^*_c(M/F)$. Then  the metric on the leaves of $F$ does not have
positive scalar curvature.  

If  $F$ is Riemannian, the condition on the Novikov-Shubin invariant becomes $NS(\maD) > \codim(F)/2$.
\end{conjecture}

The road to our results is rather long, and as a result the proof given in  \cite{BH19} is rather dense.  In this paper, we give more background material and {explain how to extend our non-existence results to non-compact manifolds. We also take this opportunity to  give a simplified conceptual interpretation  of the Gromov-Lawson relative index theorem, based on our extension of Haefliger cohomology to non-compact manifolds and the superconnection heat approach of \cite{HL99}.}

\section{Characteristic classes, the $\what{A}$ genus, and Scalar Curvature}
For more on the facts quoted here see \cite{ MS74, KN99}.

Denote by $M$ a smooth  manifold and by $C^{\infty}(M)$ the smooth (real or complex valued) functions on $M$.  The tangent bundle to $M$ is denoted  by $TM$.  Suppose that $E \to M$ is a smooth (real or complex) vector bundle, and denote its smooth sections by  $C^{\infty}(E)$.   A connection $\nabla$ on $E$ assigns to each vector field $X \in C^{\infty}(TM)$ a smooth map $\nabla_X:C^{\infty}(E) \to C^{\infty}(E)$ so that for any $f \in C^{\infty}(M)$,  $\alpha_1, \alpha_2  \in  C^{\infty}(E)$, and $X,Y \in C^{\infty}(TM)$,
$$
\nabla_{fX+Y} (\alpha_1) \,\, = \,\,   f \nabla_X(\alpha_1) +\nabla_Y( \alpha_1)
\quad \text{and} \quad
\nabla_X(f\alpha_1 + \alpha_2) \,\, = \,\,   X(f)\alpha_1 + f \nabla_X(\alpha_1) +\nabla_X( \alpha_2).
$$
Connections always exist (use local triviality of the bundle $E$ and a partition of unity argument), are  actually local operators, and $\nabla_X(\alpha)(x)$ depends only on $X(x)$, and the first derivatives of $\alpha$ at $x$.  Suppose  $\mfa =(\alpha_1,...,\alpha_k)$ is a local framing of $C^{\infty}(E)$.  Then the local connection form is the $k \times k$ matrix of local one forms $\theta = [\theta_{ij}]$ defined by 
$$
 \nabla_{X}(\alpha_i) \,\, = \,\, \sum_j \theta_{ij}(X) \alpha_j.
$$

The curvature operator $R$  of $\nabla$ associates to $X,Y \in  C^{\infty}(TM)$ the smooth map $R_{X,Y}: C^{\infty}(E) \to C^{\infty}(E)$ given by 
$$
R_{X,Y}(\alpha) \,\, = \,\, (\nabla_X \nabla_Y \,\, - \nabla_Y\nabla_X - \nabla_{[X,Y]})(\alpha).
$$
The curvature is a pointwise operator, that is   $R_{X,Y}(f\alpha) = fR_{X,Y}(\alpha)$, and its local form is the $k \times k$ matrix of local two forms $\Omega = [\Omega_{ij}]$ given by 
$$
 \Omega_{ij} \,\, = \,\,   d \theta_{ij} \,\, - \,\, \sum_k  \theta_{ik} \wedge \theta_{kj}.
$$
The wonderful property of $\Omega$ is that, while it depends on the local framing $\mfa$, it changes very simply with a change of framing.  In particular, if $\mfa  =  g \cdot \what{\mfa}$ is a change of framing, where $g \in \GL_k$ (real or complex), then $\Omega = g \what{\Omega} g^{-1}$.  So, for example, $\tr(\Omega)$ is a globally well defined (closed!) 2 form on $M$, so defines an element in $H^*(M;\R)$, the deRham cohomology of $M$.

Now suppose that $E$ is a real bundle, and consider the differential forms  $c_1(\Omega)$,..., $c_k(\Omega)$ defined by
 $$
 \det( I - \frac{\lambda}{2i\pi} \Omega) \,\, = \,\,  1 \,\, + \,\, \lambda c_1(\Omega)  \,\, + \,\,  \lambda^2 c_2 (\Omega)   \,\, + \,\,\cdots  \,\, + \,\, \lambda^k c_k(\Omega).
$$
Then,
\begin{itemize}

\item  $c_j(\Omega)$ is a well defined, closed $2j$ form on $M$.

\medskip
\item  $c_{2j+1}(\Omega)$ is exact, so $[c_{2j+1}(\Omega)] = 0$ in $H^{ 4j+2}(M;\R)$.

\medskip
\item The $j$\,th Pontrjagin class of $E$ is:    \hspace{0.5mm}  $p_j(E) = [c_{ 2j}(\Omega)] \in H^{4j}(M;\R)$. 

\medskip
\item  $p_j(E)$ does NOT depend on $\nabla$, only on $E$
\end{itemize}
If $E$ is a $\C$ bundle, then we get the same results, except that the $c_{2j+1}(\Omega)$ are not necessarily exact.
\begin{itemize}
\item 
The $j$\,th Chern class of $E$ is:  \hspace{0.5mm}  $c_j(E) = [c_{j}(\Omega)] \in H^{2j}(M;\R)$.

\medskip
\item Both the Chern and  Pontrjagin classes are the images in  $H^*(M;\R)$ of integral classes, i.e.\ classes in $H^*(M;\Z)$.  This is why the term $1/2i\pi $ appears in their definitions.

\end{itemize}
Set  $C_k(\sigma_1,...,\sigma_k) \,\, = \,\, \sum_{j =1}^k  e^{x_j}$, where $\sigma_i$ is the $i$\,th elementary symmetric function in $x_1,...,x_k$.
\begin{itemize}
\item   The Chern character of $E$ is:  \hspace{0.5mm} $\ch(E) = C_k(c_1(E),...,c_k(E))  \,\,=\,\, [ \tr (\exp(- \Omega/2i\pi))].$
\medskip
\item  Note that the $\ch(E) \,\, = \,\, \dim(E) \,\, + $ higher order terms.
\end{itemize}
Now assume that $M$ is compact and Riemannian, and denote the classical $K$-theory of $\C$ vector bundles over $M$ by $K^0(M)$.  Then
\begin{theorem}
$\ch:K^0(M)\otimes \R \to H^{2*}(M;\R)$ is an isomorphism of algebras.
\end{theorem}

The $\what{A}$ genus $\what{A}(M)$ is an exceptionally important invariant of $M$.   
Set
$$
\what{A}(\sigma_1(x_i), \sigma_2(x_i), \sigma_3(x_i),...)  \,\, = \,\,\prod_i \frac{\sqrt{x_i}/2}{\sinh(\sqrt{x_i}/2)}   \,\, = \,\, \prod_i \Big[\Big( \sum_{k=0}^{\infty} \frac{x_i^k}{2^{2k}(2k+1)!}\Big)^{-1}\Big].
$$  
Then
\begin{itemize}
\item
$\dd\what{A}(TM)  \,\, = \,\,  \what{A}(p_1(TM), p_2(TM), ...) \in H^{4*}(M;\Q)$.  \,\,Note that $\what{A}(TM)) \,\, = \,\, 1 \,\, + $ higher order terms.
\medskip
\item
$\dd\what{A}(M)  \,\, = \,\, \int_M \what{A}(TM) \in \Q$.
\end{itemize}
The invariant $\what{A}(M)$ is defined for all compact manifolds and is non-zero only if $\dim(M) = 4k$.   If $M$ is a spin manifold, it is an integer, but this is not true in general, e.g. $\what{A}(\C P^2) = -1/8$.  For more on this see below.

Denote the Riemannian structure on $M$ (the inner product on $TM$)  by $\langle \cdot, \cdot \rangle$.  Then $TM$ admits a special connection, the Levi-Civita connection $\nabla$, with curvature $R$.   The local matrices for these operators, computed with respect to local orthonormal bases, are $so_n$ (the Lie algebra of $SO_n$, so skew symmetric) matrices of forms.  The scalar curvature $\kappa$ is one of the simplest invariants of $(M,\langle \cdot, \cdot \rangle)$ and it is given as follows.  Suppose that $X_1,...,X_n$ is a local orthonormal framing of $TM$.  Then 
$$
\kappa(x) \,\, = \,\, -\sum_{i,j= 1}^n \langle R_{X_i,X_j}(X_i), X_j \rangle(x).
$$  
$ \langle R_{X_i,X_j}(X_i), X_j   \rangle(x)$ is the classical curvature at the point $x$ of the local two dimensional sub-manifold which is the exponential of a neighborhood of $0 \in \span(X_i,X_j)$.  $\kappa$ does not depend on the framing used, so it  is ``the average of the local  curvatures at $x$".  It  is a smooth function on $M$, and is the weakest of all curvature invariants.  

Note that, for any compact Riemannian manifold $M$, $M \times \S^2$ admits a metric of positive scalar curvature.  Just take the given metric on $M$ product with a multiple of the standard  metric on $\S^2$ whose  curvature swamps $\kappa$ on $M$.  This construction still  works if $M$ is non-compact, provided $\kappa$ on $M$ is bounded below.

\section{ Spin Structures and the Atiyah-Singer Operator}\label{spin}

For more on the facts quoted here see \cite{MS74, LM}.

Every Lie group has a universal covering group.  Of course the universal cover of $SO_2 = \S^1$ is $\R$. However, for $n > 2$, the universal covering of $SO_n$ is a double covering,  denoted $\Z_2 \to Spin_n \to SO_n$.    $M$ is a spin manifold if it is orientable and there is principle $Spin_n$ bundle $P$ over $M$ so that $TM \simeq P \times_{Spin_n} \R^n$.    Simply put this means  that given a trivialization of $TM$ over an open cover $\cU = \{U_i\}$, $TM \, | \, U_i \simeq U_i \times \R^n$, with change of coordinate functions $g_{ij} : U_i \cap U_j \to SO_n$, the $g_{ij}$ can be lifted to $\what{g}_{ij } : U_i \cap U_j \to  Spin_n$ so that the relations $g_{ij} g_{jk} = g_{ik}$ are preserved.

In terms of characteristic classes, $M$ is spin if and only if the first two Stiefel-Whitney classes, $w_1$ and $w_2$, of $TM$ are zero.   The first being zero means that $M$ is orientable, and the second that there are spin structures on $TM$.

A fundamental theorem in spin geometry is the following.
\begin{theorem} [A. Lichnerowicz, \cite{Li63}]
Suppose $M$ is a compact spin manifold and $\what{A}(M) \neq 0.$  
Then $M$  does not admit a metric of positive scalar curvature.
\end{theorem}

If $M$ is spin, with spin structure  $TM \simeq P \times_{Spin_n} \R^n$, then  $TM$ has  $Spin_n$ connections (constructed using trivializations whose changes of coordinates take values in $Spin_n$), so let  $\nabla$ be one.  For $n$ even, $Spin_n$ has two different irreducible representations denoted $S_{\pm}$.  The spinor super (meaning $\Z_2$ graded) bundle is the bundle over $M$, 
$$
\maS = (P \times_{Spin_n} S_{+})  \oplus (P \times_{Spin_n} S_{-})=  \maS_{+}  \oplus \maS_{-}.
$$
  The connection $\nabla$ induces a connection, also denoted $\nabla$, on $\maS$  preserving $\maS_{+}$ and  $\maS_{-}$.   The bundle $TM$ acts on $\maS$ (by Clifford multiplication), interchanging $\maS_{+}$ and  $\maS_{-}$.  For $X \in TM$,  this action is denoted by $X \cdot$.  

The  Atiyah-Singer (super) operator on $\maS$ is given as follows.  Choose a local orthonormal  framing  $X_1, \ldots, X_n$ of $TM$, and set 
$$
\maD  \,\, = \,\, \sum_{i=1}^n X_i \cdot \nabla_{X_i} \quad\quad   \text{ which may be written as }  \quad\quad  \maD  \,\, = \,\, \left[\begin{array}{cc} 0 & \maD_{-}\\
\maD_{+}& 0\end{array}\right],
$$
since $\maD$ interchanges $\maS_{+}$ and  $\maS_{-}$.  $\maD$ does not depend on the choice of framing.

The operator $\maD$, so also $\maD_{+}$ and  $\maD_{-}$, is  elliptic.  This means roughly that it differentiates in all directions.  An important consequence is that if $M$ is compact, then the kernel and co-kernel of $\maD$ are finite dimensional.   In particular, both  $\ker(\maD_{+})$ and $\ker(\maD_{-})$ are finite dimensional.  The index $\Ind(\maD)$ of $\maD$ is the integer 
$$
\Ind(\maD) = \dim(\ker(\maD_{+})) - \dim(\ker(\maD_{-})).
$$
One of the most important theorems of the 20th century is the Atiyah-Singer Index Theorem, \cite{ASI}.  In this particular case it is 
\begin{theorem}\label{ASTHM}
\hspace{10mm}$\dd \Ind(\maD)   \,\, = \,\, 
\what{A}(M)  \,\, = \,\, \int_M \what{A}(p_1(TM), p_2(TM), ...)$.
\end{theorem}
This is the reason $\what{A}(M) \in \Z$ for spin manifolds.   The question of just why this was true was one of the motivations for the Atiyah-Singer Index Theorem.

\medskip\noindent
{\bf Proof of the Lichnerowicz Theorem.}    The bundle $\maS$ has an inner product $ \langle \cdot, \cdot \rangle$ on it.  The operator $\nabla$ has an adjoint $\nabla^*$ defined by  $ \langle  s_1, \nabla_X^*s_2 \rangle = \langle  \nabla_X s_1 , s_2 \rangle$, for $s_1,s_2 \in C^{\infty}(\maS)$.  In \cite{S32, Li63}, Schr\"{o}dinger and, independently, Lichnerowicz proved that the Atiyah-Singer operator $\maD$ and the connection Laplacian $\nabla^*\nabla$ on $\cS$ are related by 
$$
\maD^2 \,\, = \,\, \nabla^\ast\nabla + \frac{1}{4}\kappa.
$$  
Now, suppose $\kappa > 0$ and $\maD(s) = 0$.  Then for all $x \in M$, 
$$ 
0 \,\, = \,\, \langle \maD^2 s, s \rangle (x)  \,\, = \,\, \langle \nabla^\ast\nabla s, s \rangle (x) \,\, + \,\, \frac{1}{4}\kappa(x) \langle s, s \rangle (x) \,\, = \,\, 
$$
$$
\langle \nabla s, \nabla s \rangle (x) \,\, + \,\, \frac{1}{4}\kappa(x) \langle s, s \rangle (x) \,\, \geq \,\, \frac{1}{4}\kappa (x)\langle s, s \rangle (x)\,\, \geq \,\, 0.
$$

\medskip

Since $\kappa (x) > 0$, $\langle s,   s \rangle (x) = 0$ for all $x$.   So $s = 0$, and  $\ker(\maD) =  \ker(\maD_{+}) \oplus  \ker(\maD_{-}) = 0$.    Thus 
$$
\hspace{2.4cm}  \what{A}(M) \,\, = \,\, \Ind(\maD) \,\, = \,\, 0.  \hspace{2.4cm} \qed
$$

\medskip

{\bf Question}:  What about the torus $\T^n$?  Does it admit metrics of PSC?  It is a Lie group, so its tangent bundle is trivial and all of its  characteristic classes are zero, in particular, $\what{A}(\T^n)=0$, and nothing above applies.

\section{Enlargeability and PSC: Gromov \& Lawson's results }\label{GLresults}

Suppose that $E \to M$ is a  $\C$ bundle with a unitary (i.e.\ a $U_k$) connection $\nabla^E$.  Then the operator $\maD$ extends to $\maD_E$ acting on $\maS \otimes E$,  and it is given by 
$$
\maD_E   \,\, = \,\, \sum_{i=1}^n X_i \cdot (\nabla_{X_i} \otimes I + I \otimes \nabla^E_{X_i}),
$$
where $X_1,..., X_n$ is an orthonormal  basis of $TM$.
Denote by $R^E$  the curvature of $\nabla^E$.   For  a section $s \otimes \alpha$ of $\maS \otimes E$,  set
$$
\maR^E (s \otimes \alpha) \,\, = \,\, \frac{1}{2} \sum_{i,j=1}^n (X_i \cdot X_j \cdot s) \otimes R^E_{X_i,X_j}(\alpha).
$$ 
Then the relation $\maD^2 \,\, = \,\, \nabla^\ast\nabla + \frac{1}{4}\kappa$ generalizes to 
 $$
 \maD_E^2 \,\, = \,\, \nabla^*\nabla \,\, + \,\, \frac{1}{4} \kappa \,\, + \,\, \maR^E.
 $$   
 
\begin{definition}\label{contract} A smooth map 
$f:M \to M'$ between Riemannian manifolds is $\ep$ contracting if $||f_*(X)|| \leq \ep ||X||$ for all $X \in TM$.
\end{definition}

\begin{definition}\label{enlarge}
A Riemannian $n$-manifold  is enlargeable if for every  $\ep >0$, there is a metric covering $\what{M}$ of $M$ and   $f_\ep:\what{M} \to \S^n(1)$ (the usual $n$ sphere of radius $1$) which is $\ep$ contracting, constant near $\infty$ (i.e. outside a compact subset), and non-zero degree.
\end{definition}

In a series of papers \cite{GL1, GL2, GL3}, Gromov and Lawson, used the above relation and the notion of enlargeability to prove, among many other things, the following.
\begin{theorem}[Gromov-Lawson]
A  compact  enlargeable spin manifold does not admit any metric of positive scalar curvature.
\end{theorem}
 
As $\T^n$ is enlargeable and spin it does not admit a metric of PSC.  In particular, $\R^n$ is its universal metric cover which has contracting maps to $\S^n(1) $, namely send the open $N$ ball $||x|| < N$  one to one and onto $\S^n(1) - (1,0,0,...,0)$, and $||x|| \geq N$ to $(1,0,0,...,0)$.  It is spin since all its characteristic classes, including its Stiefel-Whitney classes, are zero.

\medskip
\begin{proof}
Assume $M$ is enlargeable and has PSC.  We may also assume that $M$ has even dimension $2n$, for if not we may replace it by $M \times \S^1$, with the product of the metric on $M$ with the usual metric on $\S^1$.   If $M$ has PSC, so does $M \times \S^1$.  

\medskip
Suppose that  $\what{M}$ is a compact covering of $M$  and $f_{\ep}:\what{M} \to \S^{2n}$ is $\ep$ contracting, with $\ep$ to be determined.  Then $\what{M}$ is spin,  since the Stiefel-Whitney classes of a cover are the pull-back of those of the base, and it  also has PSC, since that is a local property.   We denote the scalar curvature of $\what{M}$ by $\kappa$ also since it is the pull back of $\kappa$  on $M$.  Choose  a $\C$ bundle $E \to \S^{2n}$ with  $\dd \int_{\S^{2n}}  c_n(E) \neq 0$.  This is possible because $\ch:K^0(\S^{2n})\otimes \R \to H^{*}(\S^{2n};\R)=  H^{0}(\S^{2n};\R) \oplus H^{2n}(\S^{2n};\R))= \R \oplus \R$ is an isomorphism.  Note that all the Chern classes of $E$ are zero,  except $c_0(E) = \dim(E)$ and $c_n(E)$, since $H^k(\S^{2n};\R)= 0$ for $k \neq 0, 2n$.   It turns out that  $\ch(E)=\dim(E) + \frac{1}{(n-1)!} c_n(E)$.   Set $\what{E} = f_{\ep}^\ast(E)$.  Since $f_{\ep}$ has non-zero degree, 
$$
\int_{\what{M}}  c_n(\what{E})  \,\, = \,\,  \int_{\what{M}}  c_n(f_{\ep}^*(E))  \,\, = \,\,  \dd \int_{\what{M}} f_{\ep}^* (c_n(E)) \,\, = \,\, \deg(f_{\ep})\dd \int_{\S^{2n}} c_n(E)  \neq 0.
$$
A direct calculation gives $|| \maR^{\what{E}}|| \leq C \ep^2$, where $C \in \R_+$ depends only on $E$ and the dimension of $M$, \cite{LM} p.\ 307.  Since $M$ is compact and $\kappa$ is smooth (and so uniformly positive), we can require $\ep$ to be  so small that $\frac{1}{4}\kappa +  \maR^{\what{E}} >0$.  Choose $c > 0$ so that $c\I  \leq  \frac{1}{4}\kappa + \maR^{\what{E}}$ as an operator on $\what{\maS} \otimes \what{E}$.

Now suppose that $\varphi \in \ker(\maD_{\what{E}})$ and $\varphi \neq 0$.  We may assume $||\varphi|| = 1$.  Then
$$
0 \,\, = \,\, ||\maD_{\what{E}}^2(\varphi)||    \,\, = \,\,  ||\maD_{\what{E}}^2(\varphi)|| \,   ||\varphi||   \,\, \geq \,\,  | \langle \maD_{\what{E}}^2(\varphi), \varphi \rangle  |
\,\, = \,\, 
| \int_{\what{M}} \langle \nabla^*\nabla\varphi , \varphi \rangle \,\, + \,\, 
\int_{\what{M}} \langle (\frac{1}{4}\kappa \,\, + \,\, \mathcal{R}^{\what{E}})\varphi, \varphi \rangle | \,\, \geq \,\, 
$$
$$
\int_{\what{M}} ||\nabla\varphi||^2 \,\, + \,\, 
\int_{\what{M}} \langle c\I\varphi, \varphi \rangle 
\,\, = \,\,  
\int_{\what{M}} ||\nabla\varphi||^2 \,\, + \,\, c\int_{\what{M}}||\varphi||^2  \,\, \geq \,\, c \,\, > \,\, 0,
$$
an obvious contradiction.   
So $\ker(\maD_{\what{E}})=0$.
In this  case Atiyah-Singer gives
$$
0 \,\, = \,\,   \Ind(\maD_{\what{E}}) \,\, = \,\, \int_{\what{M}} \what{A}(T\what{M}) \ch(\what{E}).
$$

Just as above, we have $\ker(\maD)=0$  for $\maD$ on $\what{M}$, so also
$\dd \int_{\what{M}} \what{A}(T\what{M})= \ind(\maD) = 0$.  Now, the cohomology classes $\what{A}(T\what{M}) = 1 + $ higher order terms and $\ch(\what{E}) = \dim(\what{E}) + \frac{1}{(n-1)!}c_n(\what{E})$.   Thus,
$$
0 \,\, = \,\, \int_{\what{M}} \what{A}(T\what{M}) \ch(\what{E})  \,\, = \,\, \dim(\what{E})\int_{\what{M}} \what{A}(T\what{M}) \,\, + \,\,   \frac{1}{(n-1)!} \int_{\what{M}} c_n(\what{E})  \,\, = \,\, 0 \,\, + \,\, \,\,    \frac{1}{(n-1)!} \int_{\what{M}} c_n(\what{E}) \,\, \neq \,\,  0,
$$
a contradiction.   

If $\what{M}$ is not compact, Gromov-Lawson employ their relative index theorem to achieve the same end.
\end{proof}

The cohomology for bounded geometry manifolds below gives a different path to the proof in the non-compact case}.  This gives us a tool to replace methods on manifolds which do not extend to foliations in general, e.g., Gromov-Lawson's relative index theory, which are crucial to the proof of Theorem \ref{BHspin}.  

\section{A cohomology for bounded geometry manifolds}

In this and the next section, we consider the case of $M$ being a Riemannian manifold of bounded geometry.   Bounded geometry means that  the injectivity radius is bounded below (essentially how far you can go before starting to come back), and that the curvature and all of its covariant derivatives are uniformly bounded (by a bound depending on the order of the derivative).  This class of manifolds includes all compact manifolds and their covers as well as any leaf (and its covers) of any foliation of a compact manifold.  There are also examples of such manifolds which cannot be a leaf of a foliation of any compact manifold, \cite{AH96, PS81}.  

A lattice $T  = \{x_i  \in M\}$ is a  countable set so that: there is $r > 0$, so that distinct elements are at least a distance $r$ apart;  $\cU = \{B(x_i, r)\}$ is a locally finite cover by the open balls $B(x_i, r)$ of radius $r$;   each  $B(x_i, r)$ is  contained in a coordinate chart;  the Lebesgue number of the cover $\cU$ is positive, i.e. there is $\lambda >0$ so that  every $B(x, \lambda)$, $x \in M$, is  a subset of some  $B(x_i, r)$.

Manifolds of  bounded geometry  always admit lattices for some $r > 0$, and bounded geometry implies that there is an integer $\ell$ so that any element of $\cU$  intersects at most $\ell$ other elements of $\cU$ non-trivially. 

For a lattice $T$, let   $\maB({T})$  be  the set of all bounded functions $f:T \to \R$. Denote by $V \subset \maB({T})$, the vector subspace  generated by elements of the form $f_{i,j}$, where $f_{i,j}(x_i) = 1$, $f_{i,j}(x_j) = -1$, and $f_{i,j}(x_k) = 0$ otherwise.  We need to take care as to what ``... the vector subspace generated by ..." means.    Denote by $d( \cdot, \cdot)$ the distance function on $M$. 

An element $f \in V$ is a possibly infinite sum  $f = \sum_{i,j}a_{i,j}f_{i,j}$, $a_{i,j} \in \R$, so that  there is a positive constant $C \in \R$, depending on $f$,  with
$$
a_{i,j} = 0 \text{  if  }d(x_i, x_j) > C   \quad   \text{and}  \quad      |a_{i,j} | \leq  C   \text{   for all   }  i,j.
$$ 
Note that the first condition implies that there is an integer $\what{\ell}$, so that for each $i$ the number of $a_{i,j} \neq 0$ is less than $\what{\ell}$.   
$\overline{V}$ is the closure of $V$ under the sup norm.  
\medskip
The BG (bounded geometry) cohomology of $(M,T)$ is  
$$
H^0_b(M/TM) =   \maB({T})/ \overline{V}.
$$    

\medskip
One way to think about this cohomology is that it consists of monetary ``states" in the sense that each point $x_i$ has a (globally bounded, both positively and negatively) bank account, and the state is unchanged by transfers among the accounts, that is $v \in  \overline{V}$, provided that there is a global bound (depending on $v$) on the size of each transfer made by $v$  and on how far apart the accounts  involved in any transfer made by $v$ are. In particular, $a_{i,j}f_{i,j}$ ``transfers" $a_{i,j}$ from $x_j$ to $x_i$.

\medskip
A priori, $H^0_b(M/TM) $ depends on the choice of lattice $T$.  However, this is not the case.

\begin{theorem}
$H^0_b(M/TM)$ does not depend on $T$, nor does it depend on $r$.
\end{theorem}

\begin{proof}  Let $T_{1}$ and $T_{2}$ be two lattices for $M$ and set $T = T_{1} \cup T_{2}.$  Now $T$ may not be a lattice as the centers of distinct balls may be arbitrarily close, but there is still an $\ell$ so that each element of $T$ meets at most $\ell$ other elements of $T$.  It addition,  $\maB({T})/ \overline{V}$ still makes sense.  With these observations, it is an easy exercise to show
$$
\maB(T_1)/ \overline{V_1}    \simeq    \maB({T})/ \overline{V} 
\simeq \maB(T_2)/ \overline{V_2}.
$$ 
Similarly, one can show that $H^0_b(M/TM)$ does not depend on $r$. 
\end{proof}


\section{A trace and an index theorem for bounded geometry manifolds}  \label{trace}

{Denote by $\maB(M)$ the bounded measurable functions on $M$. Suppose that  $M$ is oriented, then}  there is a surjective linear map  $\dd \int_F :  \maB(M) \to H^0_b(M/TM)$. This map is given by choosing a partition of unity $\{\phi_i\}$ subordinate to the cover $\cU$ associated to $T$, and setting 
$$
(\int_{F} f)(x_i)  \,\, = \,\,   {\left[\int_{B(x_i,r)} \phi_i (x) f(x)  dx\right]},
$$
where $dx$ is the volume form on $M$ and where $[\bullet]$ denotes the class of the bounded function $x_i\mapsto \int_{B(x_i,r)} \phi_i (x) f(x)  dx$ in $H_b^0(M/TM)$.  {In this section, $F=TM$ stands for the trivial top-dimensional foliation of $M$ with one leaf, we give the similar definition for general foliations in the next section}. Note that each integration $f \to \dd  \int_{B(x_i,r)}  \phi_if dx $ is essentially integration over a compact fibration, so $\dd \int_F$ satisfies the Dominated Convergence Theorem.  It is independent of the choice of $T$ and the partition of unity because it takes values in $H^0_b(M/TM)$. 

For any bundle $E \to M$, denote by $\maB(E)$ the bounded measurable sections of $E$. 
Denote by  $E \otimes E^{*}$ the bundle $\pi_{1}^{*}(E) \otimes \pi_{2}^{*}(E^*)$ over $M \times M$, where $\pi_{1}, \pi_{2}: M \times M \to M$  are the two projections, and $E^*$ is the dual bundle of $E$.  For any $k \in \maB(E \otimes E^*)$, $k(x,y)$ is a linear map  from $E_y$ to $E_x$, so  $k(x,x)$  has a well defined trace $\tr(k(x,x))$,  and  $\tr(k) \in \maB(M)$.  The BG trace $\tr_b(k)$ of $k$ is the  BG cohomology class 
$$
\tr_b(k) \, = \, \int_F \tr(k).
$$

\medskip
Denote by $C^{\infty}_{s}(E \otimes E^{*})$ the set of $k \in C^{\infty}(E\otimes E^{*})$ so that there is  $s > 0$ with $k(x,y)=0$ for all $(x,y)$ with $d(x,y) \geq s$.  In addition we require that $k$ and all its derivatives be bounded on $M \times M$.  Any $k \in C^{\infty}_{s}(E \otimes E^{*})$ defines a bounded smoothing  operator $K$ with {finite propagation} on $L^{2}(E)$, the $L^{2}$ sections of $E$.  It is given by
$$
K(\varphi)(x) \,\, = \,\, \int_M k(x,y)\varphi(y) dy, 
$$ 
and $\tr(k(x,x))$ is a bounded smooth function on $M$.  

\medskip

Let $k_{1}(x,y) \in \maB(E \otimes E^{*})$ and $k_{2}(x,y) \in C^{\infty}_{s}(E \otimes E^{*})$, and define $k_{1}\circ k_{2}$ by
$$
k_{1}\circ k_{2}(x,y) = \int_{M} k_{1}(x,z) k_{2}(z,y) \, dz,
$$
and similarly for $k_{2}\circ k_{1}$.  It is immediate that $k_{1}\circ k_{2}$ and $k_{2}\circ k_{1}$ are elements of $\maB(E \otimes E^{*})$.

\medskip

A central result of the \cite{H02} is that  $\tr_b$  is actually a trace in the usual sense.
\begin{theorem}\label{main}
Suppose $k_1 \in \maB(E \otimes E^*)$, 
and $k_2  \in C^{\infty}_s(E \otimes E^*)$.  Then
$$
\tr_b (k_1\circ k_2) \,\, =  \,\, \tr_b (k_2 \circ k_1).
$$
\end{theorem}

This property of $\tr_b$  {allows} one to extend classical results from compact manifolds to manifolds of bounded geometry.   In particular,  the Atiyah-Singer index theorem 
$$
\Ind(\maD_{\what{E}}) \,\, = \,\, \int_{\what{M}} \what{A}(T\what{M}) \ch(\what{E}) 
$$ 
of Section \ref{GLresults} extends to non-compact spin manifolds of bounded geometry.   More specifically, 
let $\maD_{E}$ be the generalized Atiyah-Singer operator associated to $E$ acting on $L^2(E)$, and $P_{E,\pm}$ the projections onto the $\ker(\maD_{E, \pm})$.  Set  
$$
\Ind(\maD_{E}) \,\, = \,\,  \tr_b(P_{E,+}) \,\, - \,\, \tr_b(P_{E,-}) \,\, \in \,\, H^0_b(M/TM).
$$

\begin{theorem}[\cite{H02}, Theorem $7$]\label{IndexJames}  Let $M$ be an even dimensional 
spin manifold of bounded geometry.   Then 
$$
\Ind(\maD_{E}) \,\, =  \,\,   \int_F \widehat{A}(M)  \ch(E)   \,\, \in \,\, H^0_b(M/TM).
$$
\end{theorem} 

The proof consists in showing that $\tr_b(e^{-t\maD_{E}^2})   \in H^0_b(M/TM) $ is independent of $t$, and then showing that $\lim_{t \to 0} \tr_b(e^{-t\maD_{E}^2}) =  \dd \int_F \widehat{A}(M)  \ch(E)$ and $\lim_{t \to \infty } \tr_b(e^{-t\maD_{E}^2}) = \Ind(\maD_{E})$.   The reader should note that this last convergence  is pointwise and not necessarily
uniform.   A standard way to obtain invariants in this situation is to apply a linear functional to $H^0_b(M/TM)$.   In doing so, one needs to check that the functional commutes with taking limits.  This does not happen in general.  For an explicit 
counter example see  \cite{Roe:1988II}, and for a systematic discussion, \cite{GI:2000}.
There are similar theorems for the other classical operators.  In addition, the celebrated Lefschetz theorem in \cite{AB:1967,  AB:1968} also extends to manifolds of bounded geometry.  See \cite{H02} for details. 


\section{Foliations}

A foliation of a manifold is a sub-bundle $TF$ of $TM$ whose sections are closed under the bracket operation.  That is, if $X_1, X_2 \in C^{\infty}(TF)$, then $[X_1, X_2] \in C^{\infty}(TF)$.  Suppose the fiber dimension of $TF$ is $p$.  Then $TF$ is the tangent bundle of a partition $F$ of $M$ into disjoint $p$ dimensional sub-manifolds called the leaves of $F$.  Locally this structure is trivial, that is $M$ has a covering by coordinate charts  so that on each chart $U$ the picture is
\begin{center}
\begin{picture}(300,220) 

\put(85,210){$U\simeq \D^p(1)   \times T  \simeq \D^p(1)   \times  \D^q(1)$} 
\put(90,35){\framebox(130,155)} 

\green{
\put(65,180){\rule{62mm}{0.1mm}} 
\put(65,156){\rule{62mm}{0.1mm}} 
\put(65,134){\rule{62mm}{0.1mm}} 
\put(65,112){\rule{62mm}{0.1mm}} 
\put(65,91){\rule{62mm}{0.1mm}} 
\put(65,69){\rule{62mm}{0.1mm}} 
\put(65,47){\rule{62mm}{0.1mm}} 
}

{
\put(150,35){\rule{0.5mm}{54.45mm}} 
\put(130,97){$T$}
\put(87,15){\rule{46mm}{0.5mm}} 
\put(61,13){$\D^p(1)$} 
 }

\end{picture}
\end{center}

The notation $\D^p (r)=\{x \in \R^p \, | \, ||x||  <r\}$, and similarly for $\D^q(r)$.  The transversal $T \simeq  \{0\} \times \D^q(1)$  has dimension $q = n-p$, the co-dimension of $F$.  Coordinates on $U$ are given by $(x_1,...,x_p)$ on $\D^p(1)$, the leaf coordinates,  and $(y_1,...,y_q)$ on $\D^q(1)$, the transverse coordinates.

An open locally finite cover  $\{(U_i, \psi_i)\}$ of $M$  by  coordinate charts $\psi_i: U_i \to  \D^p(1) \times \D^q(1) \subset \R^n$  is a good cover for $F$ provided that
\begin{enumerate}
 \item 
For each $y \in \D^q(1), P_y = \psi_i^{-1}(\D^p(1) \times  \{y\})$ is contained in a leaf of $F$. $P_y$ is called a plaque of $F$.
\smallskip
\item 
If $\overline{U}_i \cap  \overline{U}_j \neq \emptyset$,  then $U_i \cap  U_j \neq \emptyset$, and $U_i \cap  U_j$ is connected.
\smallskip
\item 
Each $\psi_i$ extends to a diffeomorphism $\psi_i: V_i \to \D^p(2) \times \D^q(2)$, so that the cover $\{(V_i, \psi_i)\}$  satisfies $(1)$ and $(2)$, with $\D^p(1)$ replaced by $\D^p(2)$.
\smallskip
\item 
Each plaque of $V_i$  intersects at most one plaque of $V_j$  and a plaque of $U_i$ intersects a plaque of $U_j$ if and only if the corresponding plaques of $V_i$ and $V_j$ intersect.
\smallskip
\item
There are global positive upper and lower bounds on the norms of each of the derivatives of the $\psi_i$.
\smallskip
\item The  set $\maT = \bigcup_i \{ \psi_i^{-1}((0,0)) \}$ is a lattice for $M$. 

\end{enumerate}
Good covers always exist on compact manifolds, as well as manifolds of bounded geometry.  In addition to assuming that $M$ is of bounded geometry, we assume that the leaves of $F$ are uniformly of bounded geometry.  

While every foliation is quite simple locally, its global structure can be quite complicated:  leaves can be dense; all the leaves of the foliation can be compact sub-manifolds, but without a bound on their volumes; leaves can limit on themselves in interesting ways (called resiliant leaves); leaves which are close in one chart, may diverge greatly as they move through other charts.  This last means the the homotopy graph (see below) of the foliation may not be a fibration, which makes doing analysis on it difficult, but not insurmountable.

One of the most famous foliations is the Reeb foliation of  $\S^3$.  The three sphere is the union of two solid tori, where they are glued together along their boundaries, the two torus $\T^2$, by a diffeomorphism which interchanges the two  obvious generators of $\pi_1(\T^2)$.  The solid tori are foliated as in the following picture.  The interior leaves are $\R^2$s and the boundary $\T^2$ is also a leaf.  

\begin{center}
\includegraphics[scale=0.620]{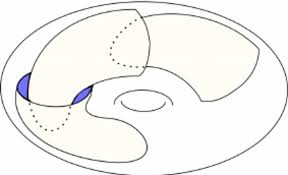}
\end{center}

One of the important  features of this foliation of $\S^3$ is that its homotopy graph (see below) is not Hausdorff.  

For more on foliations,  see \cite{L74} and \cite{CC00}. 

\section{Holonomy and Haefliger cohomology for foliations}\label{Haef}

The holonomy of a foliation associates to a leafwise path $\gamma:[0,1] \to M$ a local map $h_{\gamma}$ from any transversal through $\gamma(0)$ to any transversal through $\gamma(1)$. See the picture below for the special case where $\gamma(0) \in T_i$ and $\gamma(1) \in T_j$.  It takes the point $z \in T_i$ to the point $h_{\gamma}(z) \in T_j$ which is obtained by sliding $z$ along a path in its leaf parallel to $\gamma$.

\begin{center}
\begin{picture}(320,180) 

\put(73,170){$U_i$} 
\put(10,45){\framebox(130,115)} 

\put(263,165){$U_j$} 
\put(200,30){\framebox(130,110)} 

\green{
\put(0,150){\rule{120mm}{0.1mm}} 
\put(0,126){\rule{120mm}{0.1mm}} 
\put(0,104){\rule{120mm}{0.1mm}} 
\put(0,82){\rule{120mm}{0.1mm}} 
\put(0,61){\rule{120mm}{0.1mm}} 
\put(0,39){\rule{120mm}{0.1mm}} 
}

\put(79.5,103){$z$}
\put(73.528,102){$\bullet$}  
\put(270,103){$h_{\gamma}(z)$}
\put(263.55,102){$\bullet$} 

{
\put(75,45){\rule{0.5mm}{40.5mm}} 
\put(80,50){$T_i$} 
\put(265,30){\rule{0.5mm}{38.7mm}}
\put(270,45){$T_j$}
}

{
\put(21,64){$\alpha(0)$} 
\put(238,64){$\alpha(1)$}
\put(170,64){$\alpha$}
\put(40,59){$\bullet$}
\put(231,59){$\bullet$}  
\put(43.528,60.75){\rule{66mm}{0.4mm}}
}

{
\put(48,130){$\gamma(0)$} 
\put(264,130){$\gamma(1)$}
\put(168,132){$\gamma$}
\put(67,124.3){$\bullet$}
\put(257,124.3){$\bullet$} 
\put(70,126){\rule{66mm}{0.4mm}}
}

{
\put(164,92){$h_{\gamma}$}
\put(80,85){\rule{56mm}{0.4mm}} 
\put(238,85.8){\vector(1,0){5}} 
}

\end{picture}
\end{center}

\vspace{-0.7cm}

Assume for the moment that $M$ is a compact manifold. The (reduced) Haefliger cohomology of  $F$, \cite{Hae}, is given as follows.   Let  ${\cU} = \{ U_i \}$ be a finite good cover of $M$ by foliation charts.  We may assume that the closures of the transversals $T_i$ are disjoint, and we set $T=\bigcup\,T_i$.   Denote by $\cA^k_c(T_i)$, the space of k-forms on $T_i$ with compact support, and with the usual $C^\infty$ topology. Set  $\cA^k_c(T) = \sum_i \cA^{k}_c(T_i)$,  and denote the exterior derivative by $d_T:\cA^k_c(T)\to  \cA^{k+1}_c(T)$.   Denote by $\cH$ the pseudogroup  consisting of  the holonomy maps induced by  $F$ on $T$.  
We assume that the range of each $h_{\gamma} \in \cH$ is maximal.     A holonomy map  $h_{\gamma}:T_i \to T_j$  induces the map  $h_{\gamma}^*:  \cA^{k}_c(T_j \cap h_{\gamma}(T_i))  \to  \cA^{k}_c(T_i).$  
Denote by  $\cA^k_c(M/F)$ the quotient of $\cA^k_c(T)$ by the closure of the vector subspace generated by elements of the form $\alpha-h_{\gamma}^*\alpha$ where $h_{\gamma} \in \cH$ and $\alpha\in\cA^k_c(T)$ has support contained in the range of $h_{\gamma}$.   
The exterior derivative $d_T$ induces $d_H:\cA^k_c(M/F)\to \cA^{k+1}_c(M/F)$.   
The associated cohomology theory is denoted $H^*_c(M/F)$ and is called the Haefliger cohomology of $F$.   This construction is independent of all choices made.

Note that if $F$ given by a fibration $M \to B$, then $H^*_c(M/F) = H^*(B;\R)$, the usual deRham cohomology of $B$.

Denote by $\cA^{p+k}(M)$ the space of smooth $p+k$-forms on $M$, and by $d_M$ its exterior derivative.  Assuming that the bundle $TF$ is oriented, there is a continuous open surjective linear map,
$
\dd \int_F :\cA^{p+k}(M)\longrightarrow \cA^k_c(M/F)
$
which commutes with $d_M$ and $d_H$, so it induces the map 
$ 
\dd \int_F :\oH^{p+k}(M;\R) \to \oH^k_c(M/F).
$
This map is given by choosing a partition of unity $\{\phi_i\}$ subordinate to the cover $\cU$, and for $\omega \in \cA^{p+k}(M) $, setting 
$$
\int_F \omega \,\, = \,\,   \sum_i \int_{U_i} \phi_i \omega   \in \maA^k_c(T),
$$
where $\dd \int_{U_i}$ is integration over the fibers of the projection $U_i \to T_i$.  Note that each integration $\omega \to \dd \int_{U_i} \phi_i \omega$ is essentially integration over a compact fibration, so $\dd \int_F$ satisfies the Dominated Convergence Theorem, which is important  for extending results to non-compact manifolds.

\medskip

We now drop the assumption  that $M$ is compact, but we still assume that $TF$ is oriented. 

\medskip

The extension of $H^*_c(M/F)$  to non-compact manifolds of bounded geometry is given as follows.   Let  ${\cU}$ be a  good cover of $M$ by foliation charts as above.  By condition $(5)$ of a good cover, the $U_i$ are of a fixed size.    By condition $(6)$, there is a global bound on the number of $U_j$ intersecting any given $U_i$  non-trivially.  As above,  the transversals $T_i \subset U_i$ may be assumed to be disjoint.  The  space $\cA^k_c(T)$ consists of all smooth k-forms on  $T = \bigcup_i T_i$ which have compact support in each $T_i$, and such that they are uniformly bounded in the usual $C^\infty$ topology.   As above, we have the exterior derivative $d_T:\cA^k_c(T)\to  \cA^{k+1}_c(T)$, and the integration  $\dd \int_F:\cA^{p+k}_b(M) \to \cA^k_c(T)$, where $\cA^{p+k}_b(M)$ is the space of smooth uniformly bounded, in the usual $C^\infty$ topology,  $p+k$-forms on $M$.  As above, this induces $\dd \int_F:\oH^{p+k}_b(M;\R) \to \oH^k_c(M/F)$, which satisfies the Dominated Convergence Theorem.

The holonomy pseudogroup  $\cH$ still acts on $\cA^k_c(T)$, and $\cA^k_c(M/F)$ is still the quotient of $\cA^k_c(T)$ by the closure of the vector subspace $V$ generated by elements of the form $\alpha_{\gamma}-h_{\gamma}^*\alpha_{\gamma}$ where $h_{\gamma} \in \cH$ and $\alpha_{\gamma}\in\cA^k_c(T)$ has support contained in the range of $h_{\gamma}$.  As before, we need to take care as to what ``... the vector subspace $V$ generated by elements of the form $\alpha_{\gamma}-h_{\gamma}^*\alpha_{\gamma}$ ..." means.  This is especially important in the proof of Lemma 3.12 of \cite{H95}.   

Members of $V$ consist of possibly infinite sums of elements of the form $\alpha_{\gamma}-h_{\gamma}^*\alpha_{\gamma}$, with the following restrictions, which may depend on the given member. \\ 
\indent $\bullet$ Any $h_{\gamma}$ occurs at most once in any member.\\
\indent $\bullet$ There is a uniform bound on the length norm of all the elements $\gamma$ in any given member. The length norm of $\gamma$ is the infimum of the lengths of the leafwise paths in the equivalence class $\gamma$.    \\ 
\indent $\bullet$  There is a uniform bound in the  $C^\infty$ topology on all the terms $\alpha_{\gamma}-h_{\gamma}^*\alpha_{\gamma}$ in  the given member.   \\
Note that the second condition immediately implies that there is an integer $\ell$ so that the number of elements of the member having the domain of $h_{\gamma}$ contained in a given $T_i$ is less than $\ell$.  Also, note that these restrictions are just the translation to foliations of the restrictions imposed in the definition of BG cohomology for bounded geometry manifolds.\\

\section{ Homotopy graph ${\cG}$ of F}

We will be interested in  foliations which have Hausdorff homotopy graphs, because their graphs are close enough to being fiber bundles that it is possible to extend analysis to them which works for fiber bundles.  See the comments in Section  \ref{IndForFols}.
 
The homotopy graph ${\cG}$ of  a foliation $F$ is the groupoid which consists of equivalence classes of leafwise paths in $M$.  Two paths are equivalent if  they are homotopic in their leaf with their end points fixed.  The groupoid structure of $\cG$ is given by composition of paths, and the units $\cG_0$ are just the classes of constant paths.   If $x \in M$ we denote by  $\overline{x}$ the  class of constant path at $x$.  So, $x \to \overline{x}$ gives an embedding  $M \simeq \cG_0 \subset \cG$, by which we consider $M$ as a subset of $\cG$.

$\cG$ is a (possibly non-Hausdorff-e.g.\ the Reeb foliation) manifold of dimension $2p+q = p+ n$.   Suppose that $ (x^i_1,...,x^i_p, y^i_1,...,y^i_q)$, are coordinates on $U_i$.  Similarly for $U_j$.  Consider the set $(U_i,\gamma,U_j)$, where $\gamma(0) \in U_i$ and $\gamma(1) \in U_j$,  consisting of all classes of leafwise paths $\alpha$ starting in $U_i$, ending in $U_j$ which are parallel to $\gamma$.  This set has coordinates given by $(x^i_1,...,x^i_p, y^i_1,...,y^i_q, x^j_1,...,x^j_p)$, where $(x^i_1,...,x^i_p, y^i_1,...,y^i_q)$ are the coordinates of $\alpha(0)$, and $(x^j_1,...,x^j_p)$ are the leaf coordinates of $\alpha(1)$.
Note that, because of our assumption on the maximality of the range of  $h_{\what{\gamma}}^*:T_i \to T_j$,  $h_{\what{\gamma}}^*$ will be the same for all $\what{\gamma} \in (U_i, \gamma, U_j)$.

There are two natural maps $s,r :{{\cG}}\to M$:  $s\bigl([\alpha]\bigr)=\alpha(0)$  $r\bigl([\alpha]\bigr)=\alpha(1).$  There are two natural foliations on $\cG$, $F_s$ whose leaves are $\widetilde{L}_x = s^{-1}(x)$ the ``fibers"  of $\cG$, and $F_r$, which we more or less ignore, whose leaves are $r^{-1}(x)$.  The map $r: \widetilde{L}_x \to L_x$ is  the simply connected  cover of the leaf $L_x$.  The structures on the leaves of $F$ may thus be lifted to structures on the leaves of $F_s$, and all local properties are preserved.

\section{Index Theory}

The Atiyah-Singer Index Theorem, \cite{ASI},  is one of the watershed results of the last century.    For an elliptic differential operator on a compact manifold,  this theorem establishes the equality of the analytical index of the operator (the dimension of the space of solutions of the operator minus the dimension of the space of solutions of its adjoint) and the topological index (which is defined in terms of characteristic classes associated to the operator and the manifold it is defined over).   It subsumes many other important theorems (e.\ g.\  the Signature Theorem, the Riemann-Roch Theorem) as special cases, and it has many far reaching extensions: to families of operators; to operators on covering spaces; to operators defined along the leaves of foliations; and to operators defined purely abstractly.

  A paradigm for  such an operator is the Atiyah-Singer twisted (super) operator $\maD_E$ acting on  $\maS \otimes E = (\maS_+ \otimes E) \oplus (\maS_- \otimes E)$ of Section \ref{GLresults}.    Note that $\maD_E^2 = \maD_{E,_-}\maD_{E,_+} \oplus \maD_{E,_+}\maD_{E,_-}$ preserves the splitting of $\maS \otimes E$.  In addition,  $\maD_{E,_+}$ and $\maD_{E,_-}$ are adjoints of each other.  It is not difficult to show that this implies that
$$
\Ind(\maD_E) = \dim(\ker(\maD_{E,_-}\maD_{E,_+})) - \dim(\ker(\maD_{E,_+}\maD_{E,_-})).
$$ 
Recall the extension of  Theorem \ref{ASTHM}. \\
{\bf Theorem \ref{ASTHM} [Atiyah-Singer]}.
\hspace{10mm}$\dd \Ind(\maD_E)   \,\, = \,\,  \int_M \what{A}(TM)\ch(E)$.

\begin{proof}  The original proof of Atiyah and Singer outlined in \cite{AS63} is based on cobordism theory, and a proof along those lines appeared in \cite{P65}.   The proof in  \cite{ASI} uses psuedodifferential operators and K-theory, techniques which generalize to many interesting cases.  The proof outlined below, given by Atiyah, Bott, and Patodi in \cite{ABP} and independently by Gilkey in \cite{Gilkey}, is based on the heat equation, which is a variation of the zeta-function argument due to Atiyah and Bott. 

If $M$ is compact, the spectrums of the operators $\maD^2_+  = \maD_{E,_-}\maD_{E,_+}$  and $\maD^2_-  = \maD_{E,_+}\maD_{E,_-}$ are particularly nice.  They have the same discrete, real eigenvalues $0 = \lambda_0< \lambda_1<  \lambda_2 < \cdots,$ 
 which march off to infinity rather quickly.  In addition their eigenspaces $\maE^{\pm}_{j}$  are finite dimensional, and for $\lambda_j > 0$ they are the same dimension.  Thus we may think of  the $\maD^2_{\pm}$ as infinite diagonal matrices
$$
\maD^2_{\pm} = \mbox{Diag}(0, \ldots, 0, \lambda_1, \ldots, \lambda_1, \lambda_2, \ldots, \lambda_2,\ldots),
$$
where each $\lambda_i$ occurs only a finite number of times.  The only difference between the two matrices is the number of zeros at the beginning!  The associated heat operators are the infinite diagonal matrices 
$$
e^{-t\maD^2_{\pm} }  \,\,=\,\,   \mbox{Diag}(1,\ldots, 1, e^{-t\lambda_1}, \ldots, e^{-t\lambda_1}, e^{-t\lambda_2}, \ldots).
$$
The $\lambda_i$ go to infinity so fast that these operators are of trace class, that is 
$$
\tr  (e^{-t\maD^2_{\pm} }) \,\,=\,\,  \sum^\infty_{j=0}  e^{-t\lambda_j} \dim \maE^{\pm}_j(\lambda_j) < \infty.
$$
So 
$$
\Ind(\maD_E)   \,\,=\,\,     \dim(\maE^+_0) - \dim (\maE^-_0)     \,\,=\,\,       \tr (e^{-t\maD^2_+}) - \tr  ( e^{-t\maD^2_-}).
$$

The heat operator $e^{-t\maD^2_+}$ is much more than trace class.   In fact it is a smoothing operator, so there is a smooth section $k_{t}^{+}(x,y)$ of  $\pi^*_1\maS_+ \otimes  \pi^*_2\maS_+^*$ over  $M \times M$ (just as in Section \ref{trace}), the Schwartz kernel of $e^{-t\maD^2_+}$, 
so that for $s \in C^{\infty}(\maS_+)$,
$$
e^{-t\maD^2_+}(s)(x)=\int _{_{_{M}}} k_{t}^{+} (x,y) s(y) dy.
$$
In particular, if  $\xi^+_{j,\ell}$ is an orthonormal basis of  $\maE^+_j$, we have
$$
k_{t}^{+} (x,y) = \sum_{j, \ell} e^{-t \lambda_j} \xi^+_{j,\ell}(x) \otimes  \xi^+_{j,\ell}(y), 
$$
where the action of $\xi^+_{j,\ell}(x) \otimes  \xi^+_{j,\ell}(y)$ on $s(y)$ is 
$$
\xi^+_{j,\ell}(x) \otimes  \xi^+_{j,\ell}(y)(s(y)) \,\, = \,\,    < \xi^+_{j,\ell}(y),s(y)> \xi^+_{j,\ell}(x),   
$$
and $< \cdot, \cdot>$ is the inner product on $\maS_{+,y}$.
It follows fairly easily that 
$$
{\tr } (e^{-t\maD^2_+}) \,\, = \,\,  \int_{_{_{M}}}{\tr }\Bigl{(}k_{t}^{+}(x,x)\Bigr{)}dx.
$$
Similarly for $e^{-t\maD^2_-}$.  So we have that 
$$
\Ind(\maD_E) \,\, = \,\, 
\int_{_{_{M}}}{\tr }\Bigl{(}k_{t}^{+}(x,x)\Bigr{)}dx -  
\int_{_{_{M}}}  {\tr }\Bigl{(}k_{t}^{-}(x,x)\Bigr{)}dx,
$$
which is independent of $t$.
For $t$ near zero, the heat operator is essentially a local operator and so is subject to local analysis.  It is a classical result, see for instance \cite{GilkeyBook} and \cite{BGV}, that it has an asymptotic expansion as $t \to 0$.   In particular,  for $t$ near 0,  where $n$ is the dimension of the compact manifold,
$$
{\tr } (k^{\pm}_t(x,x)) \,\, \sim\,\,  \sum_{m\ge -n} t^{m/2} a^{\pm}_m (x),
$$
where the $a^{\pm}_m(x)$ can be computed locally, that is, in any coordinate system and relative to any local framings.   Each $a^{\pm}_m(x)$ is a complicated expression in the derivatives of the $\maD^2_{\pm}$, up to a finite order which depends on $m$.   Now we have, since $\Ind(\maD_E)$ is independent of $t$,
$$
\Ind(\maD_E) \,\, = \,\, \lim_{t \to 0} \Big( \int_{_{_{M}}}{\tr } \sum_{m\ge -n} t^{m/2} a^{+}_m (x)dx -  
\int_{_{_{M}}}  {\tr } \sum_{m\ge -n} t^{m/2} a^{-}_m (x)dx \Big) \,\, = \,\, 
\int_{_{_{M}}} a^+_0(x) - \int_{_{_{M}}} a^-_0(x).
$$
It was the hope, first raised explicitly by McKean and Singer \cite{McS67}, that there might be some ``miraculous" cancellations in the complicated expressions for the $a^{\pm}_0(x)$ that would yield the Atiyah-Singer integrand, that is there would be a local index theorem.   Atiyah-Bott-Patodi and Gilkey, \cite{ABP,Gilkey},  showed that this was indeed the case, at least for twisted Dirac operators.  That is 
$$
\int_{_{_{M}}} a^+_0(x) - \int_{_{_{M}}} a^-_0(x) \,\, = \,\, \int_{M} \what{A}(TM) \ch(E).
$$
For a particularly succinct proof, which shows that the cancellations are not at all ``miraculous", but rather natural, see \cite{Getzler}. Standard arguments in K-theory, then lead directly to the full Atiyah-Singer Index Theorem. 
\end{proof}

\section{Index theory for foliations}\label{IndForFols}

In this section we give an outline of a heat equation proof for the extension of the higher families index theorem to foliations of compact manifolds. The original proof, by other methods, is due {to} Connes, \cite{ConnesBook}.

The Atiyah-Singer index theorem for families is a bit more complicated than that for compact manifolds.  Suppose that $F \to M \to B$ is a fiber bundle with fibers $F$, and $F$, $M$ and $B$ are compact manifolds.  Suppose further that $\maD_E$ is a fiberwise Atiyah-Singer operator associated to  a bundle $E$ over $M$.  Then,  on each fiber $\maD_E$ is elliptic, and  after some work, we may assume that the finite dimensional vector spaces $\ker(\maD^2_{\pm} )(x)$, $x \in B$, actually amalgamate to give smooth vector bundles over $B$.   So they have well defined Chern characters $\ch(\ker(\maD^2_{\pm} )) \in H^*(B;\R)$, and by definition,  $\Ind(\maD_E)  =  \ch(\ker(D^2_+))  - \ch(\ker(D^2_-))$.   Denote the tangent bundle along the fiber of $M \to B$ by $TF$, and integration over the fibers by $\dd \int_F$.  

\begin{theorem} [Atiyah-Singer families index theorem, \cite{ASIV}]
$$
\Ind(\maD_E) \,\, = \,\, \ch(\ker(\maD^2_+)) \,\, - \,\, \ch(\ker(\maD^2_-))  \,\, =  \,\, \int_F \what{A}(TF)\ch(E).
$$
\end{theorem}

One major problem of working with foliations, as opposed to fibrations, is that, in general, a foliation $F$ will have both compact and non-compact leaves.  This introduces a number of difficulties, for example non-compact leaves can limit on compact ones, causing fearsome problems with the transverse smoothness of the heat operators.   Some of these difficulties can be solved by working on the homotopy graph $\cG$ of $F$ instead of $F$ itself, provided that $\cG$ is Hausdorff. 

The possible non-compactness of the leaves of $F_s$ also causes problems, particularly with the spectrums of leafwise operators, since on even the simplest non-compact manifold, namely $\R$, the spectrum of the usual Laplacian is the entire interval $[0,\infty)$.  Thus we can not think of the heat operators as nice infinite dimensional diagonal matrices with entries going quickly to zero.  However, these heat operators are still smoothing, so have nice smooth Schwartz kernels when restricted to any leaf.

As noted above, if $\cG$ is Hausdorff, then it is almost (but maybe not) a fiber bundle.  In particular, if $K \subset \wL_x = s^{-1}(x)$ is a compact subset of a leaf of $F_s$, there is an open  neighborhood $U_K$ of $x \in M$ and an open neighborhood $V_K $ of $K \subset \wL_x$, so that $s^{-1}(U_K)$ contains an open set  $W_K \simeq U_K \times V_K$, so that $\{y\} \times V_K \subset s^{-1}(y)$ for all $y \in U_K$.  That is, there is a local product structure near $K$.  This is enough so that Duhamel's formula for the derivative of a family of heat kernels extends to heat kernels defined on the leaves of $F_s$, see \cite{H95}, which is enough to extend standard results to this case.  

The heat kernel we are interested in is that for a Bismut superconnection for $F$. The notion of a superconnection was introduced by Quillen \cite{Q85}.  In \cite{Bis86}, Bismut constructed such a superconnection on fibrations and used it to give a heat equation proof of the Atiyah-Singer index theorem for families of compact manifolds.  In \cite{H95}, following Berline and Vergne, \cite{BV}, a Bismut superconnection  $B_t$ was constructed for foliations.  In simplest form it is $B_t = \sqrt{t} \maD_E + \nabla$, where $\nabla$ is a Bott connection, \cite{B70},  on the normal bundle $TM/TMF$ of $F$.   It is then pulled back to $\cG$ by the natural map $r:\cG \to M$.   The resulting operator,  think $\wB_t  = \sqrt{t} \wtit{\maD}_{\wtit{E}} + \wtit{\nabla}$,  is not a leafwise operator on the leaves of $F_s$.  However, its square, $\wB^2_t$, is a leafwise operator,  and is a super operator, that is $\Z_2$ graded just as $\maD^2_E$ is.  Thus we can write $\wB^2_t = \wB^2_{t,+}\oplus \wB^2_{t,-}$.  The heat kernel we want is $e^{-\wB^2_t}$.  

One of the results of \cite{H95} is that the leafwise  Schwartz kernel $k_t(\gamma_1,\gamma_2) = k_{t,+}(\gamma_1,\gamma_2) \oplus k_{t,-}(\gamma_1,\gamma_2)$ of   $e^{-\wB^2_t}$ is smooth in all its variables, both leafwise and  transversely to the leaves.  This allows us to define a Chern character which takes values in the Haefliger cohomology $H^*_c(M/F)$.    In particular,  this Chern character is given by the Haefliger class of the super trace  $\tr_s(e^{-\wB^2_t})$, namely 
$$
\ch(e^{-\wB^2_t})  \,\, = \,\,  [\tr_s(e^{-\wB^2_t})]  \,\, = \,\, [\int_F \tr_s (k_t(\overline{x},\overline{x})dx]  \,\, = \,\,
 [\int_F \tr(k_{t,+}(\overline{x},\overline{x}))]  -  [\int_F \tr(k_{t,-}(\overline{x},\overline{x})) dx].
 $$
We define 
$$
\Ind(\maD_E) \,\, = \,\, \ch(e^{-\wB^2_t})  \,\, \in \,\, H^*_c(M/F).
$$  

The proof  in \cite{HL99}  of the families index theorem for foliations  has three steps.  The first is to show that $\tr_s(e^{-\wB^2_t})$ is a closed Haefliger form and its cohomology class is independent of $t$, that is, the argument above for compact manifolds still holds.   This is the main result of \cite{H95}.  The fact that it is closed relies heavily on Duhamel's formula and the trace property of $\tr_s$.  Its independence of $t$ is fairly standard. 

The second step is to compute the limit as $t \to 0$ of $\tr_s(e^{-\wB^2_t})$.  The calculation for families of compact manifolds in \cite{Bis86} works just as well for foliations because the Schwartz kernel of $e^{-\wB^2_t}$ becomes very Gaussian along the diagonal as $t \to 0$, so the result is purely local, and locally the foliation case looks just like the compact families case.   So  for $\wB_t$ associated to $\maD_E$, 
$$
\lim_{t \to 0} \tr_s(e^{-\wB^2_t}) \,\, = \,\,  \int_F \what{A}(TF)\ch(E).
$$ 

Of course the final step is to compute the limit as $\lim_{t \to \infty}\tr_s(e^{-\wB^2_t})$, which is the main result of \cite{HL99}.   

The leafwise operator $\wtit{\maD}^2_{\wtit{E}}$  on $(\cG, F_s)$, satisfies the hypothesis of the Spectral Mapping Theorem, \cite{RS80},  so it has nice spectral projection operators.   Denote by $P_0$  the graded projection onto $\ker(\wtit{\maD}^2_{\wtit{E}})$, that is $P_0 = P_{0,+} \oplus P_{0,-}$ where $P_{0,\pm}$ is the projection onto $\ker(\wtit{\maD}^2_{\pm})$.   So $P_0$ is a super ``bundle", just as in the families case.

The Novikov-Shubin invariant  $NS(\maD_E)$ measures the spectral density of $\wtit{\maD}_{\wtit{E}}^2$ near zero.  Denote by $P_{\ep}$ the  spectral projection of $\wtit{\maD}_{\wtit{E}}^2$ for the interval $(0,\ep)$.  Then
$$
NS(\maD_E)  \,\, = \,\, \beta \,\, > \,\, 0  \quad  \Longleftrightarrow \quad \tr_s(P_{\ep}) \,\, \sim \,\, \ep^{\beta} \text{   as   } \ep \to 0.
$$
Note that $\tr_s(P_{\ep}) \in \oH^*_c(M/F)$.  The statement  $\tr_s(P_{\ep}) \,\, \sim \,\, \ep^{\beta} \text{   as   } \ep \to 0$ means that there is  a  constant $C>0$, so that for each $\ep$ near zero there is an element $Q_{\ep} \in \tr_s(P_{\ep})$,  with    $\|  Q_{\ep}  \|_T \leq C\epsilon^\beta$.   $NS(\maD_E)$ is the sup over all such $\beta$.  
Here $\| \cdot \|_T$ is the uniform norm, induced from the metric on $M$, on forms on the transversal $T$. The fact that for any two good covers of $M$ by foliation charts there is an integer $\ell$ so that any plaque of the first cover intersects at most $\ell$ plaques of the second cover implies easily that this condition does not depend on the choice of good cover.  
Note that the larger  $NS(\maD_E)$ is, the sparser the spectrum of $\wtit{\maD}^2_{\wtit{E}}$ is near $0$.

The Schwartz kernels of the spectral projections $P_0$, and $P_{\ep}$ are always smooth on each leaf .  They are transversely smooth if their Schwartz kernels are smooth in all directions, both leafwise and transversely.  

\begin{theorem}[\cite{HL99,BH08}]\label{HLBH}
Suppose that  the homotopy graph of $F$ is Hausdorff, the spectral projections $P_0$, and $P_{\ep}$ are transversely smooth (for $\ep$ sufficiently small), and that  $NS(\maD_E) > 3\codim(F)$.   Then
$$
\lim_{t \to \infty}  \tr_s(e^{-t\wtit{B}_t^2})  \,\, = \,\,  \tr_s(e^{-(P_0 \wtit{\nabla} P_0)^2})    \,\, = \,\,  
 \tr(e^{-(P_{0,+} \wtit{\nabla} P_{0,+})^2})  \,\, - \,\,  \tr(e^{-(P_{0,-} \wtit{\nabla} P_{0,-})^2}).
$$

If  $F$ is Riemannian, the condition on the Novikov-Shubin invariants becomes $NS(\maD_E) > \codim(F)/2$.
\end{theorem}
A very brief outline of the proof is given in Section  \ref{proofHLBH}.  

The condition on the Novikov-Shubin invariants for Riemannian foliations is sharp.  See \cite{BHW}.  The operator  $P_0 \wtit{\nabla} P_0$ is a ``connection" on the ``index bundle", that is on the super ``bundle" $\ker(\wtit{\maD}^2_{\wtit{E}})$.  Set 
$$
\ch(\ker(\maD_{E}) )    \,\, = \,\,  \tr_s(e^{-(P_0 \wtit{\nabla}  P_0)^2}).  
$$

The main result of \cite{HL99} is
\begin{theorem}  \label{HLBH2}
Suppose that  the homotopy graph of $F$ is Hausdorff, $\maD_E$ is transversely regular near $0$, that is the spectral projections $P_0$, and $P_{\ep}$ are transversely smooth (for $\ep$ sufficiently small), and $NS(\maD_E) > 3 \codim(F)$.   Then  in  $H^*_c(M/F)$,
$$
\Ind(\maD_E) \,\, = \,\, \ch(e^{-\wB^2_t})  \,\, = \,\,   \int_F \what{A}(TF)\ch(E)     \,\, = \,\,   \ch(\ker(\maD_{E}) ).
$$
\end{theorem}

\begin{corollary}
Suppose that  the homotopy graph of $F$ is Hausdorff.
If there is a gap in the spectrum of $\wtit{\maD}_{\wtit{E}}^2$ at  zero, that is there is $\ep > 0$ with $P_{\ep} = 0$,  and $P_0 = 0$, then 
 $$ 
 \int_F \what{A}(TF)\ch(E)  \,\, = \,\, \Ind(\maD_{E})   \,\, = \,\,  0.
 $$
 \end{corollary}
This is immediate as transverse smoothness is obvious, $NS(\maD_E) = \infty$ in this case, and $\ch(\ker(\maD_E) ) = 0$.
 
\section{Proof of Theorem \ref{BHspin}}\label{PfBHspin}

Recall\\ 
{\bf Theorem  \ref{BHspin}} \cite{BH19}
 {\it Suppose that $F$  is a spin foliation with Hausdorff homotopy graph of a compact enlargeable  manifold $M$.   Then $M$  does not admit a metric which induces positive scalar curvature on the leaves of $F$.}

We now have the ingredients to prove this.  

We may assume that both $M$ and $F$ are even dimensional, say $2n$ and $2p$.   If $F$ is even dimensional and $M$ is not, we  replace them by $F$  and  $M \times \S^1$, where the new foliation $F$ is just the old $F$ on each $M \times \{x\}$ for $x \in S^1$.  If  $F$ is not even dimensional, we replace it by $F \times \S^1$ and $M$ by $M \times \S^1$ or $M \times \S^1 \times \S^1 = M \times \T^2$, depending on whether $M$ is not or is even dimensional.

Since $M$ is enlargeable,  it has a covering $\wM \to M$ and an $\ep$  contracting map $f_{\ep}: \wM \to \S^{2n}$, with $\ep$ as small as we like.   So  $f_{\ep}$ is constant outside a compact subset and has non-zero degree.  Thus there is a Hermitian bundle $\wE \to \wM$, so that $\wE$ is trivial outside a compact subset of $\wM$, and 
$$ 
\ch(\wE) \,\, = \,\, \dim(\wE) \,\, + \,\,  \frac{1}{(n-1)!} c_n(\wE)   \quad \text{with} \quad  \int_{\wM }c_n(\wE) \neq 0.
$$
In addition,  we may assume that $|| \maR^{\wE}_{\wF} || \leq C\ep^2$, just as in Gromov-Lawson, where $\wF$ is the foliation on $\wM$ induced by $F$, and $\ep$ is as small as we like.  

Assume $M$ has a metric which induces positive (so uniformly positive since $M$ is compact) scalar curvature on the leaves of $F$.  The same is true of the lift of the metric  to $\wM$ and $\wF$.  We may choose $\ep$ so small that the leafwise operator 
$$
\wtit{\maD}_{\wE}^2 \,\, = \,\, \wtit{\nabla}^* \wtit{\nabla} \,\, + \,\,  \frac{1}{4}\wtit{\kappa} \,\, + \,\, \mathcal{R}^{\wE}_{\wF}
$$
is uniformly positive on the leaves of $\wtit{F}$,  that is there is $c > 0$ so that $c\I \leq \frac{1}{4}\wtit{\kappa} \,\, + \,\, \mathcal{R}^{\wE}_{\wF}$ as an operator.  Then, just as in the Gromov-Lawson case,  $\ker(\wD_{\wE}) = 0$, so $P_0 = 0$.  

What about $P_{\ep}$? 
Suppose that $P_{\ep} \neq 0$ for all $\ep >0$.  Let  $\varphi \neq 0$ be in the image of $P_{\ep}$.
Then there is $\wtit{x}_{\ep}  \in \wM$ so that $\varphi \neq 0$ on $s^{-1}(\wtit{x}_{\ep})$.  We may assume that the 
 $L^2$ norm $||\varphi||_{s^{-1}(\wtit{x}_{\ep})}$ of  $\varphi$ on $s^{-1}(\wtit{x}_{\ep})$ is $1$. 
Then we have  for all small positive $\ep$, 
$$
\ep \,\, \geq \,\, ||\wtit{\maD}_{\wE}^2 \varphi||_{s^{-1}(\wtit{x}_{\ep})}   \,\, = \,\,  ||\wtit{\maD}_{\wE}^2 \varphi||_{s^{-1}(\wtit{x}_{\ep})} \,   ||\varphi||_{s^{-1}(\wtit{x}_{\ep})}   \,\, \geq \,\,  | \langle \wtit{\maD}_{\wE}^2\varphi, \varphi \rangle_{s^{-1}(\wtit{x}_{\ep})} | 
\,\, = \,\, 
$$
$$
| \int_{s^{-1}(\wtit{x}_{\ep})} \langle\wtit{\nabla}^* \wtit{\nabla} \varphi , \varphi \rangle \,\, + \,\, 
\int_{s^{-1}(\wtit{x}_{\ep})} \langle ( \frac{1}{4}\wtit{\kappa} \,\, + \,\, \mathcal{R}^{\wE}_{\wF})\varphi, \varphi \rangle | \,\, \geq \,\, 
\int_{s^{-1}(\wtit{x}_{\ep})} ||\wtit{\nabla}\varphi||^2 \,\, + \,\, 
\int_{s^{-1}(\wtit{x}_{\ep})} \langle c\I \varphi, \varphi \rangle 
\,\, \geq  \,\,  c||\varphi||^2_{s^{-1}(\wtit{x}_{\ep})} \,\, \geq \,\, c,
$$
an obvious contradiction.
Thus the spectrum of $\wD^2_{\wE}$ has a  gap at zero, { {i.e.\ it is contained in $[\ep,\infty)$ for some $\ep >0$}.}

Denote by $\I^m$, $m = \dim(\wE)$,  the trivial $\C$ bundle over $\wM$ of complex dimension $m$.  {By the same argument,  $\wtit{\maD}_{\I^m}^2$ has a gap in its spectrum at zero, so again the Schwartz kernel of the projection onto the kernel of $\wtit{\maD}_{\I^{m}}$ is the zero section. Exactly as in \cite{H02}, there is an index theorem which holds in the bounded smooth Haefliger cohomology $H_c^*(\wM/\wF)$ for each of the two operators. Roughly speaking, we would like to apply this theorem and interpret the difference of the two formulaes as for a single manifold, an equality in some integrable Haefliger cohomology, and integrate it  against the Connes fundamental class to get a contradiction. This general principle is indeed  legal in the spirit of a relative index theorem for foliations, and is easily justified  in our specific case where we have gaps in the two spectra of the involved operators $\wtit{\maD}_{\wE}^2$ and $\wtit{\maD}_{\I^m}^2$. This was briefly indicated  in \cite{BH19} and  we proceed now to give more details for the convenience of the reader.}

%


{As in \cite{Roe91} and \cite{HL99}, one uses the heat equation method. The key point is the local computation of the local supertrace of the heat kernel  of the Bismut superconnection associated with our foliation \cite{H95}. One has though to be careful as usual since the heat kernels don't provide compactly supported forms, so that one first needs to approximate it by using  Schwartz functions with compactly supported Fourier transforms where the usual transgression formula  holds inside compactly supported differential forms and where the independence of $t$ in the final answer is obvious from the closedness of the Connes transverse fundamental cycle $C$. So we follow here the {{skeleton}} of the proof as given in  \cite{Roe91} for  a single manifold and we define the functional calculus using the inverse Fourier formula now in the context of {{superconnections}}.  {{This uses that the Grassmann}} variables are nilpotent, see for instance \cite{GorokhovskyLott} and also \cite{BC} for similar constructions. So we may approximate the Gauss function exactly as in \cite{Roe91} and use the unit speed principle for the two Dirac operators to deduce that the difference of the pointwise supertraces associated with $\wtit{\maD}_{\wE}$ and $\wtit{\maD}_{\I^{m}}$ is a compactly supported smooth transverse differential form, denote it $\omega_t$, and hence yields a compactly supported closed Haefliger form  after integration over the leaves. Moreover, we can find such Schwartz function so that when $t\to 0$, the assymptotics of the local supertraces coincides with those of the corresponding heat  superconnections, see again \cite{Roe91}.}

{To be more specific, denote by $dvol_{\wF}$  the leafwise volume form associated with the fixed metric,
then the closed compactly supported Haefliger differential form ${{\dd x \longmapsto \int_{\wF}}} \omega_t  dvol_{\wF} (x)$
can be integrated against the Haefliger fundamental cycle, say the Connes fundamental closed current $C$, to give the $t$-independent scalar
$$
\left<\, C, \;  \int_{\wF} \omega_t dvol_{\wF} \, \right>,
$$
which also does not depend on the choice of the Schwartz function, see again \cite{Roe91}. 
The  large time limit vanishes, i.e.
$$
\lim_{t\to +\infty} \left<\, C, \;  \int_{\wF} \omega_t dvol_{\wF} \, \right> = 0,
$$
because it coincides with the difference of the indices which both vanish here, and we have gaps around $0$ in the spectra \cite{Roe91}. On the other hand and as explained above the small time limit is given by the same small time limit of the heat kernels, hence it is precisely given by
$$
\lim_{t\to 0} \left<\, C, \;  \int_{\wF} \omega_t dvol_{\wF} \, \right> = \left< C,  \int_{\wF} \hat{A} (F) [\ch (\wE) - m] \right> = \left< C,  \int_{\wF} \ch_n (\wE)\right> = \int_{\wM} \ch_n (\wE).
$$
{{So we conclude that $\dd \int_{\wM}\ch_n(\wE) = 0$ which gives the needed contradiction,  since$\dd\int_{\wM} \ch_n (\wE) = \frac{1}{(n-1)!} \int_{\wM} c_n (\wE)$ is nonzero. }}
}
%
%
 

\medskip
{Now all these considerations show that the above proof remains valid for non-compact $M$, provided we assume that $(M, F)$ has bounded geometry and the leafwise scalar curvature of the induced metric on $F$ is {{ {\em uniformly}}}  positive. Indeed,  the main arguments are  either the local asymptotics of the heat kernels, the vanishing of the large time limit, or the reduction to compactly supported Haefliger cohomology and hence reduction to estimates over compact subspaces}.  In particular we have  {{proven}},

\medskip

\begin{theorem}\label{BHspin3}
Suppose that $F$  is a spin foliation with Hausdorff homotopy graph of a non-compact enlargeable Riemannian manifold $M$, with $(M,F)$ of bounded geometry.  Then the induced metric on $F$ does not have uniformly positive leafwise scalar curvature.
\end{theorem}

\medskip

\section{Proof of Theorem \ref{BHspin2}}\label{NONPSC}

Theorem \ref{BHspin3} does not preclude $M$ from having metrics of PSC, nor do the theorems of Gromov-Lawson and Zhang on the non-existance of  uniformly positive scalar curvature.
For such a result we have the following. {Recall that $H^*_c(M/F)$ denotes the bounded Haefliger cohomology.}\\

\noindent
{\bf Theorem   \ref{BHspin2}} 
{\it Suppose that $F$  is a spin foliation with Hausdorff homotopy graph of a non-compact Riemannian manifold $M$.  Suppose  that $(M,F)$ is of bounded geometry, and that $\what{A}(F) \neq 0$ in $H^*_c(M/F)$.  Suppose further that the Atiyah-Singer operator $\maD$ on $F$ {is transversely regular near $0$
with $NS(\maD) > 3 \codim(F)$}.  Then  the metric on the leaves of $F$ does not have positive scalar curvature.}

{\it If  $F$ is Riemannian, the condition on the Novikov-Shubin invariants becomes $NS(\maD) > \codim(F)/2$.}

\medskip

{Recall again that $\maD$ is transversely regular near $0$ if the spectral projections $P_0$ and $P_{\ep}$ are transversely smooth, for sufficiently small $\ep$, see \cite{BH08}.}
\medskip

\begin{proof} {As already explained, the proof given in \cite{H02} for a single bounded geometry manifold can be extended to foliations to give the index formula in the bounded Haefliger cohomology.} The conditions on  $P_0$, $P_{\ep}$, and $NS(\maD)$ hence insure that  $\what{A}(F) = \Ind(\maD) = \ch(\ker(\maD))$ in $H^*_c(M/F)$, so we only need to show that if $F$ has PSC, then $\ker(\maD) = 0$.  But the argument in the proof of the Lichnerowicz theorem adapts easily, as above, to show exactly this. 

If  $F$ is Riemannian and  $NS(\maD) > \codim(F)/2$,  then again $\what{A}(F) = \Ind(\maD) = \ch(\ker(\maD))$ in $H^*_c(M/F)$.  See \cite{BH08}.
\end{proof}

\begin{corollary}
Suppose that $(M,F)$ is as in Theorem  \ref{BHspin2}, with positive scalar curvature on the leaves of $F$.  If the Novikov-Shubin invariant $NS(\maD)$ exists, then $NS(\maD) \leq 3\codim(F)$.
\end{corollary}
That is, if such an $(M,F)$ has leafwise PSC and $NS(\maD)$ exists, then the spectrum of the leafwise Atiyah-Singer operator must be sufficiently dense near $0$.

\section{Brief outline of the proof of Theorem \ref{HLBH}}\label{proofHLBH}

In \cite{H95}, it was proven in general that for spin foliations of compact manifolds, 
$$
\lim_{t \to 0 } \tr_s(e^{-\wtit{B}^2_t}) \,\, = \,\, \int_F \what{A}(TM) \ch(E),
$$
and that the Haefliger class of $\tr_s(e^{-\wtit{B}^2_t})$ does not depend on $t$.

In \cite{HL99},  the second half of the theorem was proven by adapting  an argument in \cite{BGV}, namely splitting the spectrum of $\wtit{\maD}^2_{\wE}$ into three pieces,  $0$ and the intervals $(0,t^{-a})$ and $[t^{-a}, \infty)$, for judicious choice of $a > 0$.  In \cite{BGV}, they deal with the compact families case, and so for $t$ large enough, the spectrum in the interval $(0,t^{-a})$ is the empty set.  This is not the case when working on the homotopy graph of $F$, since  in general the leaves of $F_s$ are not compact.  To handle this problem, some assumptions were made.  The first was that the spectral projections $P_0$, and $P_{\ep}$ are transversely smooth (for $\ep$ sufficiently small).  The second was  that $NS(\maD_E)  > 3\codim(F)$.   One also has the fact that on the  interval $[t^{-a}, \infty)$,  $\tr_s(e^{-\wB^2_t})$ is decaying very rapidly as $t \to \infty$.  A  lengthy and quite complicated argument then shows that 
$$
\lim_{t \to \infty} \tr_s(e^{-\wB^2_t}) = \tr_s(e^{-(P_0 \wtit{\nabla} P_0)^2}).
$$

In \cite{BH08}, we assumed that $F$ is Riemannian.  This means that the leaves of $F$ stay a fixed distance apart, and implies that the homotopy graph of $F$ is more than Hausdorff - it is actually a fiber bundle.  We again assumed that $\maD_E$ is transversely regular near $0$, but only needed  that $NS(\maD_E) > \codim(F)/2$.  We used a $K$-theory definition of the index, and showed that its Chern character in Haefliger cohomology is the same as $\tr_s(e^{-\wtit{B}^2_t})$.  We then used a more complicated operator of heat type to show that this  Chern character is equal to $\tr_s(e^{-(P_0 \wtit{\nabla} P_0)^2})$.

\newpage

\bibliographystyle{alpha}

\end{document}